\documentclass[a4paper,12pt]{article}

\usepackage[utf8]{inputenc}
\usepackage[english]{babel}
\usepackage{amsmath, amssymb, amsthm}
\usepackage{enumerate}
\usepackage[colorlinks=true,linkcolor=blue,citecolor=blue]{hyperref}

\title{Linear forms of a given Diophantine type \\ and lattice exponents.
       \thanks{ This research was supported in part by RFBR grant 18-01-00886 and also by Young Russian Mathematics award}}
\author{Oleg\,N.\,German}
\date{}

\theoremstyle{definition}
\newtheorem{definition}{Definition}

\newtheorem*{notation*}{Notation}

\theoremstyle{remark}

\newtheorem*{remark*}{Remark}

\theoremstyle{plain}
\newtheorem{theorem}{Theorem}
\newtheorem{lemma}{Lemma}

\newtheorem{corollary}{Corollary}
\newtheorem*{statement*}{Statement}
\newtheorem*{corollary*}{Corollary}

\DeclareMathOperator{\spanned}{span}
\DeclareMathOperator{\interior}{int}

\renewcommand{\phi}{\varphi}

\renewcommand{\vec}[1]{\mathbf{#1}}
\renewcommand{\geq}{\geqslant}
\renewcommand{\leq}{\leqslant}

\newcommand{\e}{\varepsilon}

\newcommand{\R}{\mathbb{R}}
\newcommand{\Z}{\mathbb{Z}}
\newcommand{\Q}{\mathbb{Q}}
\newcommand{\N}{\mathbb{N}}

\newcommand{\La}{\Lambda}

\newcommand{\cC}{\mathcal{C}}
\newcommand{\cD}{\mathcal{D}}

\newcommand{\cL}{\mathcal{L}}

\newcommand{\cP}{\mathcal{P}}

\newcommand{\cV}{\mathcal{V}}

\newcommand{\gS}{\mathfrak{S}}

\newcommand{\SL}{\textup{SL}}

\begin{document}

\maketitle

\begin{abstract}
  In this paper we prove an existence theorem concerning linear forms of a given Diophantine type and apply it to study the structure of the spectrum of lattice exponents.
\end{abstract}


\section{Introduction}

Let $\cL_d$ denote the space of full rank lattices in $\R^d$ of covolume $1$, $\cL_d\cong\SL_d(\R)/\SL_d(\Z)$. According to Mahler's compactness criterion (see \cite{cassels_GN}) the orbit $\cD_d\La$ of a lattice $\La\in\cL_d$ under the action of the group of diagonal matrices
\[\cD_d=\Big\{\textup{diag}(e^{t_1},\ldots,e^{t_d})\,\Big|\,t_i\in\R,\ \sum_{i=1}^dt_i=0 \Big\}\]
is relatively compact if and only if the function
\[\Pi(\vec x)=\prod_{1\leq i\leq d}|x_i|^{1/d},\qquad\vec x=(x_1,\ldots,x_d)\in\R^d,\]
is bounded away from zero at nonzero points of $\La$.
%
%
In case the orbit is not relatively compact, it is natural to ask how fast it can leave any given compact set, or in other words, how fast $\Pi(\vec x)$ can tend to zero as $\vec x$ ranges through nonzero lattice points. The simplest quantity describing the asymptotic behaviour of $\Pi(\vec x)$ is the \emph{Diophantine exponent} of a lattice.

\begin{definition} \label{d:lattice_exponent}
  Let $\La$ be a lattice of full rank in $\R^d$.
  The \emph{Diophantine exponent} of $\La$ is defined as
  \[\omega(\La)=\sup\Big\{\gamma\in\R\ \Big|\,\Pi(\vec x)\leq|\vec x|^{-\gamma}\text{ for infinitely many }\vec x\in\La \Big\},\]
  where $|\cdot|$ is the Euclidean norm.
\end{definition}

Equivalently,
\begin{equation*} 
  \omega(\La)=
  \limsup_{\substack{\vec v\in\La \\ |\vec v|\to\infty}}\frac{\log\big(\Pi(\vec v)^{-1}\big)}{\log|\vec v|}\,.
\end{equation*}
Clearly, Definition \ref{d:lattice_exponent} does not depend on the choice of the norm.

\paragraph{Spectrum of lattice exponents.}

One of the first questions concerning lattice exponents is what values this quantity can attain. It follows from the definition of $\Pi(\vec x)$ that for each positive $t$ we have $\omega(t\La)=\omega(\La)$. Thus, all possible values of $\omega(\La)$ are provided by $\cL_d$, so, we can define the corresponding spectrum as
\[\Omega_d=\Big\{\omega(\Lambda)\,\Big|\,\Lambda\in\cL_d \Big\}.\]
Minkowski's convex body theorem implies the trivial bound
\[\omega(\La)\geq0.\]
Another trivial observation is that $\omega(\La)=0$ whenever $\Pi(\vec x)$ is bounded away from zero at nonzero points of $\La$. For instance, this holds for any lattice of a complete module in a totally real algebraic extension of $\Q$ (see \cite{borevich_shafarevich}). Due to Schmidt's subspace theorem \cite{schmidt_subspace} we also have $\omega(\La)=0$ for a much wider class of algebraic lattices satisfying certain independence conditions (see \cite{skriganov_1998}, \cite{german_2017}). It is worth mentioning that, same as with real numbers, such an algebraic lattice behaves as an average unimodular lattice. Namely, it was shown by Skriganov \cite{skriganov_1998} that for almost every $\La\in\cL_d$ we have
\[\Pi(\vec x)^d\gg_{\La,\e}(\log(1+|\vec x|))^{1-d-\e}\quad\text{ for }\vec x\in\La\backslash\{\vec 0\}.\]
Thus, for almost every $\La\in\cL_d$ we have $\omega(\La)=0$. Later, D.\,Kleinbock and G.\,Margulis completed Skriganov's theorem to a proper multidimensional multiplicative generalization of Khintchine’s theorem (see \cite{kleinbock_margulis_1999}).

It seems 
very natural to expect
that every nonnegative value is attainable by lattice exponents, i.e. that $\Omega_d=[0,+\infty]$. However, until recently the only examples of lattices with positive finite $\omega(\La)$ known to the author were the ones described in \cite{german_2017}. Those lattices give the values
\begin{equation} \label{eq:spectrum_from_schmidt}
  \frac{\,ab\,}{cd}\,,\qquad
  \begin{array}{l}
    a,b,c\in\N, \\
    a+b+c=d.
  \end{array}
\end{equation}
The question whether $\Omega_d=[0,+\infty]$ for $d\geq3$ is still open. In the two-dimensional case it is trivially true due to the theory of continued fractions (see Section \ref{sec:2dim}). In the current paper we prove that at least starting with some positive boundary every real number is contained in $\Omega_d$.

The following statement is the main result of the paper.


\begin{theorem} \label{t:rays}
  For each $d\geq3$
  \[\bigg[3-\frac{d}{(d-1)^2}\,,\,+\infty\bigg]\subset\Omega_d\,.\]
\end{theorem}

\paragraph{Combined spectrum.}

We cannot avoid mentioning another natural question, which however we do not deal with in this paper. It concerns the structure of the \emph{combined spectrum}
\[\tilde\Omega_d=\Big\{\big(\omega(\La),\omega(\La^\ast)\big)\,\Big|\,\Lambda\in\cL_d \Big\},\]
where $\La^\ast$ denotes the dual lattice. It appears that $\tilde\Omega_d\neq[0,+\infty]\times[0,+\infty]$. It was shown in \cite{german_2017} that
\begin{equation} \label{eq:transference_for_lattice_exponents}
  \tilde\Omega_d\subset
  \left\{\big(x,y)\in[0,+\infty]^2\,\middle|\,
  \begin{array}{l}
    x\geq\dfrac{y}{(d-1)^2+d(d-2)y} \\
    y\geq\dfrac{x^{\vphantom{|}}}{(d-1)^2+d(d-2)x}
  \end{array} \right\}.
\end{equation}
Particularly, \eqref{eq:transference_for_lattice_exponents} implies that
\[\omega(\La)=0\iff\omega(\La^\ast)=0.\]

When proving Theorem \ref{t:rays} we do not control the dual lattice, so, the only nonzero pairs $\big(\omega(\La),\omega(\La^\ast)\big)$ currently known to the author are $(\omega,+\infty)$, where $\omega$ is of the form \eqref{eq:spectrum_from_schmidt}. Moreover, the corresponding examples described in \cite{german_2017} have a certain flaw, as in each of them the dual lattice has some nonzero points in the coordinate planes, so that the condition $\omega(\La^\ast)=+\infty$ is provided by a kind of degeneracy. It would be more interesting to construct lattices that are \emph{totally irrational}, i.e. such that neither the lattice, nor its dual contains nonzero points in the coordinate planes. In this context it is worth mentioning the paper \cite{technau_widmer} by N.\,Technau and M.\,Widmer which implies the existence of a totally irrational $\La$ such that $0\leq\omega(\La)\leq d$ and $\omega(\La^\ast)=+\infty$. Notice that in view of \eqref{eq:transference_for_lattice_exponents} the inequality $0\leq\omega(\La)\leq d$ for such $\La$ can be substituted by $\frac1{d(d-2)}\leq\omega(\La)\leq d$.

%

\paragraph{Structure of the paper.}

The rest of the paper is organized as follows. In Section \ref{sec:linear_forms} we remind some definitions and formulate Theorem \ref{t:dioph_type} concerning linear forms of a given Diophantine type. It is the main tool for proving Theorem \ref{t:rays} and, besides that, we deem it to be of independent interest. In Section \ref{sec:2dim} we deal with the two-dimensional case. In Section \ref{sec:deducing_main_theorem} we derive Theorem \ref{t:rays} from Theorem \ref{t:dioph_type}. The remaining Sections \ref{sec:key_parameters}--\ref{sec:explicit_construction} are devoted to establishing quantities responsible for the local order of approximation and proving Theorem \ref{t:dioph_type}.

\section{Linear forms and best approximation vectors} \label{sec:linear_forms}

Any lattice $\La$ of full rank in $\R^d$ admits a representation
\[\La=\Big\{\big(L_1(\vec z),\ldots,L_d(\vec z)\big)\,\Big|\,\vec z\in\Z^d \Big\},\]
where $L_1,\ldots,L_d$ are linearly independent linear forms in $\R^d$. It is rather difficult to control the values of all the $d$ forms at once. Controlling the values of each one of them separately is much simpler,
%
%
so, we obtain Theorem \ref{t:rays} as a corollary to Theorem \ref{t:dioph_type} below, which asserts the existence of linear forms of a given Diophantine type.

An analogue of Theorem \ref{t:dioph_type} in the simplest case of two variables (corresponding to $d=3$ in our notation) was proved in \cite{german_moshchevitin_bordeaux}. The argument we use in the current paper works in arbitrary dimension. Besides that, it allows not only controlling the rate of decay along the sequence of best approximation vectors, but also separating it slightly away from the rate of decay along all the other points.

Let us remind the definition of best approximation vectors. To this end let us set for each $\vec x=(x_1,\ldots,x_d)\in\R^d$
\[\underline{\vec x}=(x_1,\ldots,x_{d-1})\in\R^{d-1}\]
and let $\langle\,\cdot\,,\cdot\,\rangle$ denote the inner product. Let us also set
\[\pi_d=\Big\{\vec x=(x_1,\ldots,x_d)\in\R^d\,\Big|\,x_d=1 \Big\}.\]
We shall use this notation throughout the paper.

\begin{definition} \label{def:best_approx}
  Given $\pmb\alpha=(\alpha_1,\ldots,\alpha_{d-1},1)
  \in\pi_d$, consider the linear form
  \[L_{\pmb\alpha}(\vec x)=\langle\pmb\alpha,\vec x\rangle=\alpha_1x_1+\ldots+\alpha_{d-1}x_{d-1}+x_d.\]
  We say that a vector $\vec z\in\Z^d$ with nonzero $\underline{\vec z}$ is a \emph{best approximation vector} for $L_{\pmb\alpha}$ if
  \begin{equation} \label{eq:best_approx_def}
    |L_{\pmb\alpha}(\vec z)|\leq|L_{\pmb\alpha}(\vec z')|
  \end{equation}
  for each $\vec z'\in\Z^d$ such that $0<|\underline{\vec z'}|\leq|\underline{\vec z}|$ with a strict inequality in \eqref{eq:best_approx_def} if
  $0<|\underline{\vec z'}|<|\underline{\vec z}|$.
\end{definition}


It should be mentioned that Definition \ref{def:best_approx} differs a bit from the standard one. It is more common to consider the linear form $\langle\underline{\pmb\alpha},\underline{\vec x}\rangle$ and refer to $\underline{\vec z}$ as its best approximation vector if a similar condition holds for $\|\langle\underline{\pmb\alpha},\underline{\vec z}\rangle\|$, where $\|\cdot\|$ denotes the distance to the nearest integer. For our purposes it is more convenient to consider the complete vector $\vec z$.

If $\vec z$ is a best approximation vector for $L_{\pmb\alpha}$, then so is $-\vec z$. Taking a representative of each such a pair, we can order the set of the representatives so that
\[|\vec z_1|\leq|\vec z_2|\leq|\vec z_3|\leq\ldots.\]
This sequence is infinite if and only if there are no nonzero integer solutions to the equation $L_{\pmb\alpha}(\vec x)=0$.

\begin{definition} \label{def:asymp_dir}
  A unit vector $\vec u\in\R^d$ is called an \emph{asymptotic direction} for best approximation vectors for $L_{\pmb\alpha}$ if it is an accumulation point of the set
  \[\bigg\{\frac{\vec z}{|\vec z|}\,\bigg|\,\vec z\text{ is a best approximation vector for }L_{\pmb\alpha} \bigg\}.\]
\end{definition}

Clearly, the set of asymptotic directions is symmetric, so, same as best approximation vectors, asymptotic directions come in pairs. For a detailed study of this set we refer the reader to the paper \cite{german_2004}.

\begin{theorem} \label{t:dioph_type}
  Given $d\geq3$, $\beta>0$, set
  \[f_d(\beta)=(d-2)\beta^2+(2d-3)\beta.\]
  Then there is an $\pmb\alpha\in\pi_d$ such that

  \textup{(i)} for every $\vec z\in\Z^d$ which is a best approximation vector for $L_{\pmb\alpha}$ we have
  \[|L_{\pmb\alpha}(\vec z)|\cdot|\vec z|^{d-1}\asymp|\vec z|^{-f_d(\beta)};\]

  \textup{(ii)} for every $\vec z\in\Z^d$ which is not an integer multiple of a best approximation vector for $L_{\pmb\alpha}$ we have
  \[|L_{\pmb\alpha}(\vec z)|\cdot|\vec z|^{d-1}\gg|\vec z|^{-f_d(\beta)+(d-1)\beta};\]

  \textup{(iii)} the set of asymptotic directions for best approximation vectors for $L_{\pmb\alpha}$ contains exactly two (pairs of) points.

  Here the constants implied by ``\ $\asymp$'' and ``\ $\gg$'' are assumed to depend only on $d$, $\beta$, and $\pmb\alpha$.
\end{theorem}

\begin{remark*}
  As Prof.\,Moshchevitin noticed, a similar result should follow from Schmidt--Summerer's parametric geometry of numbers (see \cite{schmidt_summerer_2009}, \cite{schmidt_summerer_2013}, \cite{roy_annals_2015}) due to Roy's theorem (see \cite{roy_annals_2015}). The reason is that if $\vec z\in\Z^d$ is not an integer multiple of a best approximation vector for $L_{\pmb\alpha}$, there is a point $\vec z'\in\Z^d$ linearly independent with $\vec z$ such that
  \[\begin{cases}
      |L_{\pmb\alpha}(\vec z')|\leq|L_{\pmb\alpha}(\vec z)| \\
      |\underline{\vec z'}|\leq|\underline{\vec z}|
    \end{cases},\]
  in which case we have
  \[\lambda_2\bigg(\cC\Big(|\underline{\vec z}|\big/|L_{\pmb\alpha}(\vec z)|\Big)\bigg)\leq|\underline{\vec z}|,\]
  where
  \[\cC(Q)=\Big\{\vec x\in\R^d\,\Big|\,|\underline{\vec x}|\leq1,\ |L_{\pmb\alpha}(\vec x)|\leq Q^{-1} \Big\}\]
  and $\lambda_j\big(\cC(Q)\big)$ is the $j$-th successive minimum of $\cC(Q)$ w.r.t. $\Z^d$. Thus, any lower bound for $\lambda_2\big(\cC(Q)\big)$ gives a lower bound in the spirit of statement \textup{(ii)} of Theorem \ref{t:dioph_type}. However, Roy's theorem does not seem to immediately give any information concerning asymptotic directions for best approximation vectors, which we rely upon when deducing Theorem \ref{t:rays} from Theorem \ref{t:dioph_type} (see Section \ref{sec:deducing_main_theorem}).
\end{remark*}

\section{Case d=2} \label{sec:2dim}

The argument we use when proving Theorem \ref{t:dioph_type} essentially involves the assumption $d\geq3$. We should notice however that in the two-dimensional case everything is much simpler, as in this case we have such a powerful tool as regular continued fractions.

By Legendre's theorem every reduced rational $p/q$ that is not a convergent of $\alpha$ satisfies
\[\bigg|\alpha-\frac pq\bigg|\geq\frac1{2q^2}\,.\]
At the same time every convergent satisfies
\[\bigg|\alpha-\frac pq\bigg|\leq\frac1{q^2}\,.\]
Due to the classical relation
\[\mu(\alpha)-2=\limsup_{n\to+\infty}\frac{\ln a_{n+1}}{\ln q_n}\]
between the measure of irrationality
\[\mu(\alpha)=\sup\Big\{\gamma\in\R\ \Big|\,\big|\alpha-p/q\big|\leq|q|^{-\gamma}\text{ admits $\infty$ solutions in }(q,p)\in\Z^2 \Big\}\]
and the rate of growth of partial quotients, one can construct an $\alpha$ with any given $\mu(\alpha)\geq2$. So, taking into account the Legendre theorem mentioned above we get

\begin{theorem} \label{t:dioph_type_2dim}
  Given $\beta\geq0$, there is an $\pmb\alpha\in\pi_2$ such that

  \textup{(i)} for every $\vec z\in\Z^2$ which is a best approximation vector for $L_{\pmb\alpha}$ we have
  \[|L_{\pmb\alpha}(\vec z)|\cdot|\vec z|\asymp|\vec z|^{-\beta};\]

  \textup{(ii)} for every $\vec z\in\Z^2$ which is not an integer multiple of a best approximation vector for $L_{\pmb\alpha}$ we have
  \[|L_{\pmb\alpha}(\vec z)|\cdot|\vec z|\gg1.\]
\end{theorem}

Of course, it is senseless in this case to talk about asymptotic directions for best approximation vectors, there is always one obvious pair of them.

Theorem \ref{t:dioph_type_2dim} allows describing completely the spectrum of lattice exponents in the two-dimensional case. Given $\beta\geq0$, consider a pair of distinct linear forms $L_1$, $L_2$ provided by Theorem \ref{t:dioph_type_2dim} and the corresponding lattice
\[\La=\Big\{\big(L_1(\vec z),L_2(\vec z)\big)\,\Big|\,\vec z\in\Z^2 \Big\}.\]
Then $L_1(\vec z)\asymp|\vec z|$ if $|L_2(\vec z)|\leq1$, and $L_2(\vec z)\asymp|\vec z|$ if $|L_1(\vec z)|\leq1$, so, statement \textup{(i)} of Theorem \ref{t:dioph_type_2dim} already guarantees that $\omega(\La)=\beta$, regardless of statement \textup{(ii)}.

Thus, in the two-dimensional case statement \textup{(i)} alone trivially implies
\[\Omega_2=[0,+\infty].\]
For $d\geq3$ we have to make use of a multidimensional generalization of statement \textup{(ii)} provided by Theorem \ref{t:dioph_type}.

\section{Deducing Theorem \ref{t:rays} from Theorem \ref{t:dioph_type}} \label{sec:deducing_main_theorem}

Take an arbitrary $\beta\geq\frac{d}{d-1}$ and a linear form $L_{\pmb\alpha}$ provided for this $\beta$ by Theorem \ref{t:dioph_type}. Suppose $\pm\pmb\zeta_1$ and $\pm\pmb\zeta_2$ are the asymptotic directions for best approximation vectors for $L_{\pmb\alpha}$. Consider arbitrary linear forms $L_1,\ldots,L_d$ such that none of them is zero at $\pmb\zeta_1$, $\pmb\zeta_2$ and
\begin{equation} \label{eq:positive_norm_minimum}
  \inf_{\vec z\in\Z^d\backslash\{\vec0\}}|L_1(\vec z)\ldots L_d(\vec z)|>0.
\end{equation}
Then, if $\vec z\in\Z^d$ is a best approximation vector for $L_{\pmb\alpha}$, we have
\[|L_i(\vec z)|\asymp|\vec z|,\quad i=1,\ldots,d,\]
so,
\[|L_1(\vec z)\ldots L_{d-1}(\vec z)L_{\pmb\alpha}(\vec z)|\asymp|\vec z|^{-f_d(\beta)},\]
whereas for every $\vec z\in\Z^d\backslash\{\vec0\}$ which is not an integer multiple of a best approximation vector for $L_{\pmb\alpha}$ we have
\begin{multline*}
  |L_1(\vec z)\ldots L_{d-1}(\vec z)L_{\pmb\alpha}(\vec z)|=
  |L_1(\vec z)\ldots L_d(\vec z)|\frac{|L_{\pmb\alpha}(\vec z)|\cdot|\vec z|^{d-1}}{|L_d(\vec z)|\cdot|\vec z|^{d-1}}\gg \\ \gg
  |\vec z|^{-f_d(\beta)+(d-1)\beta-d}\geq|\vec z|^{-f_d(\beta)},
\end{multline*}
since $|L_d(\vec z)|\ll|\vec z|$ and $\beta\geq\frac{d}{d-1}$.

Thus, for
\[\La=\Big\{\big(L_1(\vec z),\ldots,L_{d-1}(\vec z),L_{\pmb\alpha}(\vec z)\big)\,\Big|\,\vec z\in\Z^d \Big\}\]
we have
\[\omega(\La)=\frac{f_d(\beta)}{d}\,.\]
Noticing that $f_d$ monotonously maps $[0,+\infty]$ onto $[0,+\infty]$ and
\[\frac{f_d\big(\frac{d}{d-1}\big)}{d}=3-\frac{d}{(d-1)^2}\]
we get the statement of Theorem \ref{t:rays}.

\begin{remark*}
  Instead of forms satisfying \eqref{eq:positive_norm_minimum} we could have taken arbitrary linear forms generating a lattice with exponent not exceeding $\frac{d-1}d\beta-1$.
\end{remark*}

\section{Key parameters} \label{sec:key_parameters}

Given a sequence $(\vec z_k)$ of points in $\Z^d$, we set for each $k$ such that $\underline{\vec z_k}\wedge\ldots\wedge\underline{\vec z_{k+d-2}}$ is nonzero
\[\begin{aligned}
    & \La_{k,l}=\spanned_{\Z}(\vec z_k,\ldots,\vec z_{k+l-1}),
      \hskip46,6mm
      l=1,\ldots,d, \\
    & \underline{\La}_{k,l}=\spanned_{\Z}(\underline{\vec z_k},\ldots,\underline{\vec z_{k+l-1}}),\ \
      D_{k,l}=\frac{\det\underline{\La}_{k,l}}{|\underline{\vec z_k}|^l}\,,
      \hskip15mm
      l=1,\ldots,d-1, \\
    & R_k=\frac1{2|\underline{\vec z_{k+1}}|\det\underline{\La}_{k,d-1}}\,,
      \hskip10.3mm
      B_k=\frac{|\underline{\vec z_{k+1}}|}{|\underline{\vec z_k}|}\,.
  \end{aligned}\]
Set also $\pmb\alpha_k$ to be the point in $\pi_d$ defined by
\[\pmb\alpha_k\perp\vec z_k,\ldots,\vec z_{k+d-2},\]
and set
\[\gS_k=\Big\{\vec x\in\pi_d\,\Big|\,|\vec x-\pmb\alpha_k|\leq R_k \Big\}.\]

We shall prove Theorem \ref{t:dioph_type} by constructing a sequence $(\vec z_k)$ which will be the sequence of best approximation vectors for a linear form. The quantities $D_{k,l}$ and $B_k$ are responsible for the local order of approximation. Knowing this, it is possible to choose parameters properly, so that the rates of decay along $(\vec z_k)$ and along sequences not intersecting $(\vec z_k)$ differ in a desired way. Lemmas \ref{l:bound_for_the_outside_of_the_lattice}, \ref{l:bound_for_z1}, \ref{l:bound_for_the_lattice} we prove in the next Section give the aforementioned connection between $D_{k,l}$, $B_k$ and the local order of approximation. We notice they hold without the assumption that $\vec z_k$ are best approximation vectors.

\section{Local order of approximation}


\begin{lemma} \label{l:bound_for_the_outside_of_the_lattice}
  Suppose $\vec z_1,\ldots,\vec z_{d-1}\in\Z^d$ are linearly independent, $|\underline{\vec z_1}|\leq|\underline{\vec z_2}|$, $\pmb\alpha\in\gS_1$, $\vec z\in\Z^d\backslash\La_{1,d-1}$.

  \textup{(i)}
  If $|\underline{\vec z}|\leq|\underline{\vec z_2}|$, then $|L_{\pmb\alpha}(\vec z)|\geq|L_{\pmb\alpha}(\vec z_1)|$.

  \textup{(ii)}
  If $|\underline{\vec z_1}|\leq|\underline{\vec z}|\leq|\underline{\vec z_2}|$, then
  $|L_{\pmb\alpha}(\vec z)|\cdot|\underline{\vec z}|^{d-1}\geq\dfrac1{2D_{1,d-1}}\,$.
\end{lemma}

\begin{proof}
  If $\pmb\alpha\in\gS_1$, then
  \[|L_{\pmb\alpha}(\vec z_1)|=
    |\langle\pmb\alpha,\vec z_1\rangle|=
    |\langle\pmb\alpha-\pmb\alpha_1,\vec z_1\rangle|=
    |\langle\underline{\pmb\alpha}-\underline{\pmb\alpha_1},\underline{\vec z_1}\rangle|\leq
    R_1|\underline{\vec z_2}|=
    \frac1{2\det\underline{\La}_{1,d-1}}.\]
  Furthermore, it follows from the definition of $\pmb\alpha_1$ that $\pmb\alpha_1\cdot\det\underline{\La}_{1,d-1}\in\Z^d$. Hence for each $\vec z\in\Z^d\backslash\La_{1,d-1}$
  \[|\langle\pmb\alpha_1,\vec z\rangle|\geq\frac1{\det\underline{\La}_{1,d-1}}\,,\]
  so, if $\pmb\alpha\in\gS_1$ and $\vec z\in\Z^d\backslash\La_{1,d-1}$, $|\underline{\vec z}|\leq|\underline{\vec z_2}|$, then
  \begin{multline*}
    |L_{\pmb\alpha}(\vec z)|=
    |\langle\pmb\alpha,\vec z\rangle|\geq
    |\langle\pmb\alpha_1,\vec z\rangle|-|\langle\pmb\alpha-\pmb\alpha_1,\vec z\rangle|=
    |\langle\pmb\alpha_1,\vec z\rangle|-|\langle\underline{\pmb\alpha}-\underline{\pmb\alpha_1},\underline{\vec z}\rangle|\geq \\ \geq
    \frac1{\det\underline{\La}_{1,d-1}}-R_1|\underline{\vec z_2}|=
    \frac1{2\det\underline{\La}_{1,d-1}}.
  \end{multline*}
  Thus,
  \[|L_{\pmb\alpha}(\vec z)|\geq
    |L_{\pmb\alpha}(\vec z_1)|,\]
  and if additionally $|\underline{\vec z}|\geq|\underline{\vec z_1}|$, then
  \[|L_{\pmb\alpha}(\vec z)|\cdot|\underline{\vec z}|^{d-1}\geq
    \frac{|\underline{\vec z_1}|^{d-1}}{2\det\underline{\La}_{1,d-1}}=
    \frac1{2D_{1,d-1}}\,.\]
\end{proof}

\begin{lemma} \label{l:bound_for_z1}
  Let $\vec z_1,\ldots,\vec z_d\in\Z^d$ be a basis of $\Z^d$.
  Then for each $\pmb\alpha\in\gS_2$ we have
  \[\frac1{D_{2,d-1}B_1^{d-1}}\bigg(1-\frac1{2B_1B_2}\bigg)\leq
    |L_{\pmb\alpha}(\vec z_1)|\cdot|\underline{\vec z_1}|^{d-1}\leq
    \frac1{D_{2,d-1}B_1^{d-1}}\bigg(1+\frac1{2B_1B_2}\bigg)\,.\]
\end{lemma}

\begin{proof}
  Since $\vec z_2,\ldots,\vec z_d$ form a primitive set of integer vectors, the vector $\pmb\alpha_2\cdot\det\underline{\La}_{2,d-1}$ is integer and primitive. Hence for each $\vec z\in\Z^d$ the value $\langle\pmb\alpha_2,\vec z\rangle$ is an integer multiple of $(\det\underline{\La}_{2,d-1})^{-1}$, and equals $\pm(\det\underline{\La}_{2,d-1})^{-1}$ for each $\vec z$ that completes $\vec z_2,\ldots,\vec z_d$ to a basis of $\Z^d$. Particularly,
  \[|\langle\pmb\alpha_2,\vec z_1\rangle|=\frac1{\det\underline{\La}_{2,d-1}}\,.\]
  Thus,
  \[|\langle\pmb\alpha_2,\vec z_1\rangle|\cdot|\underline{\vec z_1}|^{d-1}=\frac{|\underline{\vec z_1}|^{d-1}}{\det\underline{\La}_{2,d-1}}=\frac1{D_{2,d-1}B_1^{d-1}}\,.\]
  Furthermore, if $\pmb\alpha\in\gS_2$ then
  \begin{multline*}
    |\langle\pmb\alpha-\pmb\alpha_2,\vec z_1\rangle|\cdot|\underline{\vec z_1}|^{d-1}=
    |\underline{\langle\pmb\alpha}-\underline{\pmb\alpha_2},\underline{\vec z_1}\rangle|\cdot|\underline{\vec z_1}|^{d-1}\leq \\ \leq
    R_2|\underline{\vec z_1}|^d=
    \frac{|\underline{\vec z_1}|^d}{2|\underline{\vec z_3}|\det\underline{\La}_{2,d-1}}=\frac1{2D_{2,d-1}B_1^dB_2}\,.
  \end{multline*}
  It remains to make use of the inequality
  \[|\langle\pmb\alpha_2,\vec z_1\rangle|-|\langle\pmb\alpha-\pmb\alpha_2,\vec z_1\rangle|\leq
    |\langle\pmb\alpha,\vec z_1\rangle|\leq
    |\langle\pmb\alpha_2,\vec z_1\rangle|+|\langle\pmb\alpha-\pmb\alpha_2,\vec z_1\rangle|.\]
\end{proof}

\begin{corollary} \label{cor:bound_for_z1}
  Let $\vec z_1,\ldots,\vec z_d\in\Z^d$ be a basis of $\Z^d$,
  $|\underline{\vec z_1}|\leq|\underline{\vec z_2}|\leq|\underline{\vec z_3}|$. Then for each $\pmb\alpha\in\gS_2$ we have
  \[\langle\pmb\alpha,\vec z_1\rangle\cdot\langle\pmb\alpha_2,\vec z_1\rangle>0.\]
\end{corollary}

\begin{proof}
  We have $|\underline{\vec z_1}|\leq|\underline{\vec z_2}|\leq|\underline{\vec z_3}|$, i.e. $B_1\geq1$, $B_2\geq1$. Therefore, due to the lower bound provided by Lemma \ref{l:bound_for_z1} the quantity $\langle\pmb\alpha,\vec z_1\rangle$ cannot be equal to zero for $\pmb\alpha\in\gS_2$, so, $\langle\pmb\alpha,\vec z_1\rangle$ and $\langle\pmb\alpha_2,\vec z_1\rangle$ should be of the same sign.
\end{proof}

\begin{corollary} \label{cor:bound_for_z1_two_bases}
  Given $\vec z_1,\ldots,\vec z_{d+1}\in\Z^d$, suppose $\vec z_1,\ldots,\vec z_d$ and $\vec z_2,\ldots,\vec z_{d+1}$ are bases of $\Z^d$. Suppose also that
  $|\underline{\vec z_1}|\leq|\underline{\vec z_2}|\leq|\underline{\vec z_3}|\leq|\underline{\vec z_4}|$, $\pmb\alpha_3\in\gS_2$, and
  \begin{equation} \label{eq:alpha3_at_z1_and_z2}
    \langle\pmb\alpha_3,\vec z_2\rangle\cdot\langle\pmb\alpha_3,\vec z_1\rangle>0.
  \end{equation}
  Then for each $\pmb\alpha\in\gS_3$ we have
  \[\langle\pmb\alpha,\vec z_2\rangle\cdot\langle\pmb\alpha_2,\vec z_1\rangle>0.\]
\end{corollary}

\begin{proof}
  By Corollary \ref{cor:bound_for_z1} we have
  \begin{equation} \label{eq:corollary_for_Omega2}
    \langle\pmb\alpha,\vec z_1\rangle\cdot\langle\pmb\alpha_2,\vec z_1\rangle>0
    \quad\text{ for each }\pmb\alpha\in\gS_2
  \end{equation}
  and
  \begin{equation} \label{eq:corollary_for_Omega3}
    \langle\pmb\alpha,\vec z_2\rangle\cdot\langle\pmb\alpha_3,\vec z_2\rangle>0
    \quad\text{ for each }\pmb\alpha\in\gS_3.
  \end{equation}
  Since $\pmb\alpha_3\in\gS_2$, it follows from \eqref{eq:corollary_for_Omega2} and \eqref{eq:alpha3_at_z1_and_z2} that
  \[\langle\pmb\alpha_3,\vec z_2\rangle\cdot\langle\pmb\alpha_2,\vec z_1\rangle>0.\]
  Combining it with \eqref{eq:corollary_for_Omega3} gives
  \[\langle\pmb\alpha,\vec z_2\rangle\cdot\langle\pmb\alpha_2,\vec z_1\rangle>0
    \quad\text{ for each }\pmb\alpha\in\gS_3.\]
\end{proof}

\begin{lemma} \label{l:bound_for_the_lattice}
  Let $\vec z_1,\ldots,\vec z_d\in\Z^d$ be a basis of $\Z^d$,
  $|\underline{\vec z_1}|\leq|\underline{\vec z_2}|\leq|\underline{\vec z_3}|$, $\langle\underline{\vec z_1},\underline{\vec z_2}\rangle\leq0$. If $d\geq4$, let
  \begin{equation} \label{eq:bound_for_the_lattice_dets}
    \det\underline\La_{1,n}>|\underline{\vec z_2}|\det\underline\La_{1,n-1},\qquad n=3,\ldots,d-1.
  \end{equation}
  Suppose $\pmb\alpha\in\gS_2$,
  \begin{equation} \label{eq:bound_for_the_lattice_semiball}
    \langle\pmb\alpha,\vec z_2\rangle\cdot\langle\pmb\alpha_2,\vec z_1\rangle\geq0.
  \end{equation}
  Suppose $\vec z\in\La_{1,d-1}$, $|\underline{\vec z}|\leq|\underline{\vec z_2}|$, $\vec z\neq\vec 0,\pm\vec z_2$. Then

  \textup{(i)}
  we have
  \[|L_{\pmb\alpha}(\vec z)|\geq
    |L_{\pmb\alpha}(\vec z_1)|;\]

  \textup{(ii)}
  if $\vec z$ is not a multiple of $\vec z_1$, then
  \[|L_{\pmb\alpha}(\vec z)|\cdot|\underline{\vec z}|^{d-1}\geq\frac{D_{1,2}^{d-1}}{D_{2,d-1}B_1^{d-1}}\bigg(1-\frac1{2B_1B_2}\bigg).\]
\end{lemma}

\begin{proof}
  It follows from \eqref{eq:bound_for_the_lattice_dets} that each $\vec z\in\La_{1,d-1}$ such that $|\underline{\vec z}|\leq|\underline{\vec z_2}|$ lies in $\La_{1,2}$, i.e.
  \[\vec z=c_1\vec z_1+c_2\vec z_2,\quad c_1,c_2\in\Z.\]
  Besides that, since $\langle\underline{\vec z_1},\underline{\vec z_2}\rangle\leq0$, we have $|c_1\underline{\vec z_1}+c_2\underline{\vec z_2}|>|\underline{\vec z_2}|$ whenever $c_1c_2<0$. If $c_2=0$, we obviously have $|L_{\pmb\alpha}(\vec z)|\geq|L_{\pmb\alpha}(\vec z_1)|$. So, since we exclude $\pm\vec z_2$, we may assume that
  \[c_1\geq1\quad\text{ and }\quad c_2\geq1.\]
  By Corollary \ref{cor:bound_for_z1} we have
  \[\langle\pmb\alpha,\vec z_1\rangle\cdot\langle\pmb\alpha_2,\vec z_1\rangle>0\]
  for each $\pmb\alpha\in\gS_2$. Thus, if $\pmb\alpha$ satisfies additionally \eqref{eq:bound_for_the_lattice_semiball}, then either $\langle\pmb\alpha,\vec z_2\rangle$ is zero, or its sign coincides with that of $\langle\pmb\alpha,\vec z_1\rangle$. Hence, for $\vec z=c_1\vec z_1+c_2\vec z_2$ with positive integer $c_1$, $c_2$, we have
  \[|L_{\pmb\alpha}(\vec z)|=
    |\langle\pmb\alpha,\vec z\rangle|=
    c_1|\langle\pmb\alpha,\vec z_1\rangle|+c_2|\langle\pmb\alpha,\vec z_2\rangle|\geq
    |\langle\pmb\alpha,\vec z_1\rangle|=
    |L_{\pmb\alpha}(\vec z_1)|.\]
  Applying Lemma \ref{l:bound_for_z1} we get
  \[|L_{\pmb\alpha}(\vec z)|\cdot|\underline{\vec z}|^{d-1}\geq
    |L_{\pmb\alpha}(\vec z_1)|\cdot|\underline{\vec z}|^{d-1}\geq
    \frac1{D_{2,d-1}B_1^{d-1}}\bigg(1-\frac1{2B_1B_2}\bigg)
    \frac{|\underline{\vec z}|^{d-1}}{|\underline{\vec z_1}|^{d-1}}\,.\]
  It remains to observe that
  \[\frac{|\underline{\vec z}|}{|\underline{\vec z_1}|}=
    \frac{|\underline{\vec z}|\cdot|\underline{\vec z_1}|}{|\underline{\vec z_1}|^2}\geq
    \frac{|\underline{\vec z}\wedge\underline{\vec z_1}|}{|\underline{\vec z_1}|^2}=
    \frac{c_2|\underline{\vec z_2}\wedge\underline{\vec z_1}|}{|\underline{\vec z_1}|^2}=
    c_2D_{1,2}\geq D_{1,2}.\]
\end{proof}

Let us combine Lemmas \ref{l:bound_for_the_outside_of_the_lattice}, \ref{l:bound_for_z1}, \ref{l:bound_for_the_lattice} into one statement, that we shall use to prove the induction step. To do so we just need to shift the indices by a nonnegative integer $k$ and adjust the result a little bit. Then Lemmas \ref{l:bound_for_the_outside_of_the_lattice}, \ref{l:bound_for_z1}, \ref{l:bound_for_the_lattice} merge into

\begin{lemma} \label{l:three_lemmas_merged}
  Given $k\in\Z_{\geq0}$, let $\vec z_{k+1},\ldots,\vec z_{k+d}$ be a basis of $\Z^d$, $|\underline{\vec z_{k+1}}|\leq|\underline{\vec z_{k+2}}|\leq|\underline{\vec z_{k+3}}|$, $\langle\underline{\vec z_{k+1}},\underline{\vec z_{k+2}}\rangle\leq0$. If $d\geq4$, let
  \begin{equation*} 
    \det\underline\La_{{k+1},n}>|\underline{\vec z_{k+2}}|\det\underline\La_{{k+1},n-1},\qquad n=3,\ldots,d-1.
  \end{equation*}
  Let $\pmb\alpha$ be an arbitrary point in $\gS_{k+1}\cap\interior\gS_{k+2}$ satisfying
  \begin{equation} \label{eq:three_lemmas_merged_hemisphere}
    \langle\pmb\alpha,\vec z_{k+2}\rangle\cdot\langle\pmb\alpha_{k+2},\vec z_{k+1}\rangle\geq0.
  \end{equation}
  Then the following statements hold:

  \textup{(i)}
  for any $\vec z\in\Z^d$ such that $|\underline{\vec z}|\leq|\underline{\vec z_{k+2}}|$, $\vec z\neq\vec 0,\pm\vec z_{k+2}$ we have
  \[|L_{\pmb\alpha}(\vec z)|\geq
    |L_{\pmb\alpha}(\vec z_{k+1})|>
    |L_{\pmb\alpha}(\vec z_{k+2})|;\]

  \textup{(ii)} we have
  \[\frac1{2D_{k+2,d-1}B_{k+1}^{d-1}}\leq
    |L_{\pmb\alpha}(\vec z_{k+1})|\cdot|\underline{\vec z_{k+1}}|^{d-1}\leq
    \frac3{2D_{k+2,d-1}B_{k+1}^{d-1}}\,;\]

  \textup{(iii)} for any $\vec z\in\Z^d$ which is not a multiple of $\vec z_{k+1}$, is different from $\pm\vec z_{k+2}$, and satisfies
  $|\underline{\vec z_{k+1}}|\leq|\underline{\vec z}|\leq|\underline{\vec z_{k+2}}|$ we have
  \[|L_{\pmb\alpha}(\vec z)|\cdot|\underline{\vec z}|^{d-1}\geq
    \min\bigg(\frac1{2D_{k+1,d-1}}\,,
    \frac{D_{k+1,2}^{d-1}}{2D_{k+2,d-1}B_{k+1}^{d-1}}\bigg).\]
\end{lemma}

\begin{proof}
  Statements \textup{(ii)} and \textup{(iii)} follow immediately from Lemmas \ref{l:bound_for_the_outside_of_the_lattice}, \ref{l:bound_for_z1}, \ref{l:bound_for_the_lattice} in view of the inequalities $B_{k+1}\geq1$, $B_{k+2}\geq1$ provided by the assumption $|\underline{\vec z_{k+1}}|\leq|\underline{\vec z_{k+2}}|\leq|\underline{\vec z_{k+3}}|$.

  As for \textup{(i)}, the inequality $|L_{\pmb\alpha}(\vec z)|\geq|L_{\pmb\alpha}(\vec z_{k+1})|$ follows from Lemmas \ref{l:bound_for_the_outside_of_the_lattice}, \ref{l:bound_for_the_lattice}, so, it suffices to show that $|L_{\pmb\alpha}(\vec z_{k+1})|>|L_{\pmb\alpha}(\vec z_{k+2})|$.

  Suppose $\pmb\alpha\in\interior\gS_{k+2}$. Then
  \begin{multline*}
    |L_{\pmb\alpha}(\vec z_{k+2})|=
    |\langle\pmb\alpha,\vec z_{k+2}\rangle|=
    |\langle\pmb\alpha-\pmb\alpha_{k+2},\vec z_{k+2}\rangle|=
    |\langle\underline{\pmb\alpha}-\underline{\pmb\alpha_{k+2}},\underline{\vec z_{k+2}}\rangle|< \\ <
    R_{k+2}|\underline{\vec z_{k+2}}|=
    \frac1{2B_{k+2}\det\underline{\La}_{k+2,d-1}}\leq
    \frac1{2\det\underline{\La}_{k+2,d-1}}\,.
  \end{multline*}
  On the other hand, since $|\langle\pmb\alpha_{k+2},\vec z_{k+1}\rangle|=\big(\det\underline{\La}_{k+2,d-1}\big)^{-1}$,
  \begin{multline*}
    |L_{\pmb\alpha}(\vec z_{k+1})|=
    |\langle\pmb\alpha,\vec z_{k+1}\rangle|\geq
    |\langle\pmb\alpha_{k+2},\vec z_{k+1}\rangle|-
    |\langle\pmb\alpha-\pmb\alpha_{k+2},\vec z_{k+1}\rangle|= \\ =
    \frac1{\det\underline{\La}_{k+2,d-1}}-
    |\langle\underline{\pmb\alpha}-\underline{\pmb\alpha_{k+2}},\underline{\vec z_{k+1}}\rangle|>
    \frac1{\det\underline{\La}_{k+2,d-1}}-R_{k+2}|\underline{\vec z_{k+1}}|= \\ =
    \bigg(1-\frac1{2B_{k+1}B_{k+2}}\bigg)\frac1{\det\underline{\La}_{k+2,d-1}}\geq
    \frac1{2\det\underline{\La}_{k+2,d-1}}\,.
  \end{multline*}
  Hence, indeed, $|L_{\pmb\alpha}(\vec z_{k+1})|>|L_{\pmb\alpha}(\vec z_{k+2})|$.
\end{proof}

\section{Distance between the convergents} \label{sec:distance_between_alphas}

The following observation generalizes the classical statement that for two consecutive convergents $p_k/q_k$ and $p_{k+1}/q_{k+1}$ to a real number we have
\[\bigg|\frac{p_k}{q_k}-\frac{p_{k+1}}{q_{k+1}}\bigg|=\frac1{q_kq_{k+1}}\,.\]

\begin{lemma} \label{l:distance_between_alphas}
  Suppose $\vec z_1,\ldots,\vec z_d\in\R^d$ are linearly independent. Then
  \begin{equation} \label{eq:distance_between_alphas}
    |\pmb\alpha_1-\pmb\alpha_2|=\frac{\det\underline\La_{2,d-2}\cdot\det\La_{1,d}\ \ \,}
                                     {\det\underline\La_{1,d-1}\cdot\det\underline\La_{2,d-1}}\,.
  \end{equation}
\end{lemma}

\begin{proof}
  Set
  \[
    \cV_1=\spanned_\R(\vec z_1,\ldots,\vec z_{d-1}),\quad
    \cV_2=\spanned_\R(\vec z_2,\ldots,\vec z_d),\quad
    \cV_3=\spanned_\R(\vec z_2,\ldots,\vec z_{d-1}).
  \]
  Let $\cV_4$ denote the orthogonal projection of $\cV_3$ to the coordinate plane
  \[\Big\{\vec x=(x_1,\ldots,x_d)\in\R^d\,\Big|\,x_d=0 \Big\},\]
  and let $\cV_4^\perp$ be its orthogonal complement. Since it is assumed that $\pmb\alpha_1$ and $\pmb\alpha_2$ are defined correctly, we have
  \[\vec e_d=(0,\ldots,0,1)\notin\cV_1,\cV_2,\]
  whence
  \[\dim\cV_4=\dim\cV_3=d-2,\quad\dim\cV_4^\perp=2,\quad\cV_3+\cV_4^\perp=\R^d.\]       
  Denote by $\vec u_1$, $\vec u_d$ respectively the projections of $\vec z_1$, $\vec z_d$ to $\cV_4^\perp$ along $\cV_3$ and set
  \[\vec v_1=\frac{\vec u_1}{|\underline{\vec u_1}|},\qquad
    \vec v_d=\frac{\vec u_d}{|\underline{\vec u_d}|}.\]
  Denote also by $\vec v_2,\ldots,\vec v_{d-1}$ respectively the orthogonal projections of $\vec z_2,\ldots,\vec z_{d-1}$ to $\cV_4$.
  Then substitution of any $\vec z_k$ with $\vec v_k$ preserves the righthand side of \eqref{eq:distance_between_alphas}, so,
  \begin{multline*}
    \frac{|\underline{\vec z_2}\wedge\ldots\wedge\underline{\vec z_{d-1}}|\cdot|\vec z_1\wedge\ldots\wedge\vec z_d|}
         {|\underline{\vec z_1}\wedge\ldots\wedge\underline{\vec z_{d-1}}|\cdot|\underline{\vec z_2}\wedge\ldots\wedge\underline{\vec z_d}|}= \\ =
    \frac{|\vec v_2\wedge\ldots\wedge\vec v_{d-1}|\cdot|\vec v_1\wedge\ldots\wedge\vec v_d|\ \ \ }
         {|\underline{\vec v_1}\wedge\vec v_2\wedge\ldots\wedge\vec v_{d-1}|\cdot|\vec v_2\wedge\ldots\wedge\vec v_{d-1}\wedge\underline{\vec v_d}|}= \\ =
    \frac{|\vec v_2\wedge\ldots\wedge\vec v_{d-1}|^2\cdot|\vec v_1\wedge\vec v_d|}
         {|\vec v_2\wedge\ldots\wedge\vec v_{d-1}|^2\cdot|\underline{\vec v_1}|\cdot|\underline{\vec v_d}|}=|\vec v_1\wedge\vec v_d|.
  \end{multline*}
  Here we have a slight abuse of notation when we write $\underline{\vec v_1}$ and $\underline{\vec v_d}$ for the corresponding projections of $\vec v_1$ and $\vec v_d$.
  Finally, define $\pmb\alpha_1',\pmb\alpha_2'\in\pi_d$ by the conditions
  \[\pmb\alpha_1'\perp\vec v_1,\ldots,\vec v_{d-1},\quad
    \pmb\alpha_2'\perp\vec v_2,\ldots,\vec v_d,\]
  or, equivalently, by the conditions
  \[\pmb\alpha_1',\pmb\alpha_2'\in\pi_d\cap\cV_4^\perp,\quad\pmb\alpha_1'\perp\vec v_1,\quad\pmb\alpha_2'\perp\vec v_d.\]
  It is a simple two-dimensional exercise to show that
  \[|\pmb\alpha_1'-\pmb\alpha_2'|=|\det(\pmb\alpha_1',\pmb\alpha_2')|=|\det(\vec v_1,\vec v_d)|.\]
  Thus, it remains to show that
  \[|\pmb\alpha_1-\pmb\alpha_2|=|\pmb\alpha_1'-\pmb\alpha_2'|.\]
  Replacing $\vec v_d$ with $-\vec v_d$ if necessary, may assume that the pairs $\{\vec v_1,\vec e_d\}$ and $\{\vec v_d,\vec e_d\}$ are equally oriented. Then
  \[\vec v_1\wedge\vec e_d=\vec v_d\wedge\vec e_d.\]
  Therefore, for any $\vec u_2,\ldots,\vec u_{d-1}\in\R^d$ and any $\lambda\in\R$ we have
  \begin{multline*}
    \vec v_1\wedge(\vec u_2+\lambda\vec e_d)\wedge\vec u_3\wedge\ldots\wedge\vec u_{d-1}-
    \vec v_d\wedge(\vec u_2+\lambda\vec e_d)\wedge\vec u_3\wedge\ldots\wedge\vec u_{d-1}= \\ =
    \vec v_1\wedge\vec u_2\wedge\vec u_3\wedge\ldots\wedge\vec u_{d-1}-
    \vec v_d\wedge\vec u_2\wedge\vec u_3\wedge\ldots\wedge\vec u_{d-1},
  \end{multline*}
  whence it follows by induction that
  \[\pmb\alpha_1-\pmb\alpha_2=\pmb\alpha_1'-\pmb\alpha_2'.\]
\end{proof}

\begin{lemma} \label{l:nested_balls}
  Let $\vec z_1,\ldots,\vec z_d\in\Z^d$ be a basis of $\Z^d$. Suppose 
  \[|\underline{\vec z_3}|\geq|\underline{\vec z_1}|,\qquad\det\underline\La_{2,d-1}>3|\underline{\vec z_2}|\det\underline\La_{2,d-2}\,.\]
  Then
  \[\gS_2\subset\interior\gS_1,\qquad R_2<R_1/3\,.\]
\end{lemma}

\begin{proof}
  By Lemma \ref{l:distance_between_alphas} we have
  \[\frac{|\pmb\alpha_1-\pmb\alpha_2|}{R_1}=
    \frac{2|\underline{\vec z_2}|\det\underline\La_{2,d-2}}{\det\underline\La_{2,d-1}}<
    \frac23\,.\]
  On the other hand,
  \[\frac{R_2}{R_1}=
    \frac{|\underline{\vec z_2}|\det\underline\La_{1,d-1}}{|\underline{\vec z_3}|\det\underline\La_{2,d-1}}\leq
    \frac{|\underline{\vec z_2}|\cdot|\underline{\vec z_1}|\det\underline\La_{2,d-2}}{|\underline{\vec z_3}|\det\underline\La_{2,d-1}}<
    \frac13\,.\]
  Hence $|\pmb\alpha_1-\pmb\alpha_2|+R_2<R_1$, i.e. $\gS_2\subset\interior\gS_1$.
\end{proof}

\section{Basic Lemma}

Iterating Lemma \ref{l:three_lemmas_merged} we get the following statement, which is the main ingredient for proving Theorem \ref{t:dioph_type}.

\begin{lemma} \label{l:basic_lemma}
  Given $\beta>0$, let $(\vec z_k)_{k\in\N}$ be a sequence of points in $\Z^d$ such that $|\underline{\vec z_k}|\to\infty$ as $k\to\infty$.
  Suppose that for each $k$ large enough

  \textup{(a)} $\vec z_{k+1},\ldots,\vec z_{k+d}$ form a basis of $\Z^d$;

  \textup{(b)} $\langle\underline{\vec z_{k+1}},\underline{\vec z_{k+2}}\rangle\leq0$;

  \textup{(c)} $B_{k+1}\asymp D_{k+1,2}\asymp|\underline{\vec z_{k+1}}|^\beta$;

  \textup{(d)} if $d\geq4$, then

  \hskip6,8mm $\det\underline\La_{k+1,n}>|\underline{\vec z_{k+2}}|\det\underline\La_{k+1,n-1}$,

  \hskip6,8mm $\det\underline\La_{k+1,n}\asymp|\underline{\vec z_{k+2}}|\det\underline\La_{k+1,n-1},\ \ \ n=3,\ldots,d-1$;

  \textup{(e)} $\langle\pmb\alpha_{k+3},\vec z_{k+2}\rangle\cdot\langle\pmb\alpha_{k+3},\vec z_{k+1}\rangle>0$.
  \\
  Then there is an $\pmb\alpha_\infty$ in $\pi_d$ such that $\pmb\alpha_k\to\pmb\alpha_\infty$ as $k\to\infty$ and
  \[
  \begin{aligned}
    & |L_{\pmb\alpha_\infty}(\vec z)|\cdot|\underline{\vec z}|^{d-1}\asymp|\underline{\vec z}|^{-f_d(\beta)}
    & \text{ for }\ \vec z\in(\vec z_k),\ \ \\
    & |L_{\pmb\alpha_\infty}(\vec z)|\cdot|\underline{\vec z}|^{d-1}\gg|\underline{\vec z}|^{-f_d(\beta)+(d-1)\beta}
    & \text{ for }\ \vec z\notin(\Z\vec z_k).
  \end{aligned}
  \]
  Here the constants implied by ``\ $\asymp$'' and ``\ $\gg$'' are assumed to depend only on $d$ and $\beta$.

  Moreover, there is $K\in\N$ such that the sequence $(\vec z_k)_{k\geq K}$ consists of consecutive best approximation vectors for $L_{\pmb\alpha_\infty}$.
\end{lemma}

\begin{proof}
  It follows from \textup{(c)} that $B_k,D_{k,2}\to\infty$ as $k\to\infty$. Let $K$ be large enough, so that for each $k\geq K-2$ along with \textup{(a)}--\textup{(e)} we have
  \begin{equation} \label{eq:B_geq_3}
    B_{k+1}>3,\ D_{k+1,2}>3.
  \end{equation}
  Then by \textup{(d)}
  \[\det\underline\La_{k+2,d-1}>3|\underline{\vec z_{k+2}}|\det\underline\La_{k+2,d-2}\]
  regardless of whether $d=3$ or $d\geq4$. Thus, we can apply Lemma \ref{l:nested_balls} and see that
  \begin{equation} \label{eq:nesting}
    \gS_{k+2}\subset\interior\gS_{k+1}\quad\text{ for every $k\geq K-2$}
  \end{equation}
  and
  \[R_k\to\infty\quad\text{ as }\quad k\to\infty.\]
  This proves that the limit point $\pmb\alpha_\infty$ is defined correctly,
  \[\{\pmb\alpha_\infty\}=\bigcap_{k\geq K}\gS_k.\]
  It also follows from \textup{(a)}, \eqref{eq:B_geq_3}, \eqref{eq:nesting}, and Corollary \ref{cor:bound_for_z1_two_bases} that for $k\geq K-2$ we have
  \[\langle\pmb\alpha,\vec z_{k+2}\rangle\cdot\langle\pmb\alpha_{k+2},\vec z_{k+1}\rangle>0
    \quad\text{ for each }\pmb\alpha\in\gS_{k+3}.\]
  Hence $\pmb\alpha_\infty$ satisfies the condition \eqref{eq:three_lemmas_merged_hemisphere} of Lemma \ref{l:three_lemmas_merged} for every $k\geq K-2$.

  By statement \textup{(i)} of Lemma \ref{l:three_lemmas_merged} the sequence $(\vec z_k)_{k\geq K}$ consists of consecutive best approximation vectors for $L_{\pmb\alpha_\infty}$.

  In order to apply statements \textup{(ii)} and \textup{(iii)} of Lemma \ref{l:three_lemmas_merged} let us estimate $D_{k+2,d-1}$. By \textup{(c)} and \textup{(d)} we have
  \begin{multline*}
    D_{k+2,d-1}=
    \frac{\det\underline\La_{k+2,d-1}}{|\underline{\vec z_{k+2}}|^{d-1}}\asymp
    \frac{|\underline{\vec z_{k+3}}|^{d-3}\cdot\det\underline\La_{k+2,2}}{|\underline{\vec z_{k+2}}|^{d-1}}=
    B_{k+2}^{d-3}D_{k+2,2}\asymp \\ \asymp
    B_{k+2}^{d-2}\asymp
    |\underline{\vec z_{k+2}}|^{(d-2)\beta}=
    \big(B_{k+1}|\underline{\vec z_{k+1}}|\big)^{(d-2)\beta}\asymp
    |\underline{\vec z_{k+1}}|^{(d-2)(\beta^2+\beta)}.
  \end{multline*}
  Consider an arbitrary $\vec z\in\Z^d$, $|\underline{\vec z}|\geq|\underline{\vec z_K}|$. Choosing $k$ so that $|\underline{\vec z_{k+1}}|\leq|\underline{\vec z}|\leq|\underline{\vec z_{k+2}}|$ and applying statements \textup{(ii)} and \textup{(iii)} of Lemma \ref{l:three_lemmas_merged} we get
  \[|L_{\pmb\alpha_\infty}(\vec z_{k+1})|\cdot|\underline{\vec z_{k+1}}|^{d-1}\asymp
    \frac1{D_{k+2,d-1}B_{k+1}^{d-1}}\asymp
    |\underline{\vec z_{k+1}}|^{-(d-2)(\beta^2+\beta)-(d-1)\beta}=
    |\underline{\vec z_{k+1}}|^{-f(\beta)},\]
  \[|L_{\pmb\alpha_\infty}(\vec z)|\cdot|\underline{\vec z}|^{d-1}\gg
    \frac1{D_{k+2,d-1}}\asymp
    |\underline{\vec z_{k+1}}|^{-(d-2)(\beta^2+\beta)}\geq
    |\underline{\vec z}|^{-f(\beta)+(d-1)\beta}.
    \hskip4mm\]
\end{proof}

\section{Explicit construction and proof of Theorem \ref{t:dioph_type}} \label{sec:explicit_construction}

Theorem \ref{t:dioph_type} is an immediate consequence of Lemma \ref{l:explicit_construction}.

\begin{lemma} \label{l:explicit_construction}
  Given $\beta>0$, there is an $\pmb\alpha\in\pi_d$ such that the sequence $(\vec z_k)$ of best approximation vectors for $L_{\pmb\alpha}$ satisfies the hypothesis of Lemma \ref{l:basic_lemma}, and the set of asymptotic directions for that sequence contains exactly two (pairs of) points.
\end{lemma}


\begin{proof}
  Set $\vec z_i$, $i=0,\ldots,d-1$, to be equal to the $(i+1)$-th column of the matrix
  \[
  \begin{pmatrix}
    0      & 1      & 0      & \cdots & 0      \\
    0      & 0      & 1      & \cdots & 0      \\
    \vdots & \vdots & \vdots & \ddots & \vdots \\
    0      & 0      & 0      & \cdots & 1      \\
    1      & 0      & 0      & \cdots & 0
  \end{pmatrix}.
  \]
  Notice that, while $\pmb\alpha_0$ is not defined, $\pmb\alpha_1$ is well defined and equals $\vec z_0$.

  Supposing $\vec z_0,\ldots,\vec z_{k+d-1}$, $k\geq0$, are defined in such a way that every $d$ consecutive vectors form a basis of $\Z^d$ and
  \begin{equation} \label{eq:z_i_leq_z_i+1}
    |\underline{\vec z_{i-1}}|\leq|\underline{\vec z_i}|,\quad i=1,\ldots,k+d-1,
  \end{equation}
  let us define $\vec z_{k+d}$.

  \paragraph{Constructing $\vec z_{k+d}$.} Let $\vec b_{k+1},\ldots,\vec b_{k+d-1}$ be the orthogonal basis of $\R^{d-1}$ such that
  \[
  \begin{aligned}
    \vec b_i\perp\underline{\vec z_{i+1}},\ldots,\underline{\vec z_{k+d-1}},\qquad & i=k+1,\ldots,k+d-2, \\
    \langle\vec b_i,\underline{\vec z_i}\rangle=|\vec b_i|^2,\qquad & i=k+1,\ldots,k+d-1.
  \end{aligned}
  \]
  In other words, the $(d-1)$-tuple $\vec b_{k+d-1},\ldots,\vec b_{k+1}$ is obtained from $\underline{\vec z_{k+d-1}},\ldots,\underline{\vec z_{k+1}}$ by the Gram--Schmidt orthogonalization process. Particularly, $\vec b_{k+d-1}=\underline{\vec z_{k+d-1}}$.

  Then for any $\vec x,\vec y\in\R^{d-1}$ the semi-open parallelepiped $\vec x+\cP$, where
  \[\cP=\bigg\{ \sum_{i=k+1}^{k+d-2}\lambda_i\vec b_i-\lambda_{k+d-1}\vec b_{k+d-1}\ \bigg|\
                0<\lambda_i\leq1,\ i=k+1,\ldots,k+d-1 \bigg\},\]
  contains exactly one point of the affine lattice $\vec y+\underline\La_{k+1,d-1}$. Therefore, it is correct to define $\vec z_{k+d}$ to be the point of $-\vec z_{k}+\La_{k+1,d-1}$ such that
  \[\underline{\vec z_{k+d}}\in\vec b+\cP,\]
  where
  \[\vec b=
    \sum_{i=k+1}^{k+d-3}3|\underline{\vec z_{i+2}}|\frac{\vec b_i}{|\vec b_i|}+
    3|\underline{\vec z_{k+d-1}}|^{1+\beta}\bigg(\frac{\vec b_{k+d-2}}{|\vec b_{k+d-2}|}-\frac{\vec b_{k+d-1}}{|\vec b_{k+d-1}|}\bigg).\]
  Then
  \begin{equation} \label{eq:the_nu's}
    \begin{array}{l}
    \underline{\vec z_{k+d}}=
    \displaystyle\sum_{i=k+1_{\vphantom{|}}}^{k+d-2}\nu_i\vec b_i-\nu_{k+d-1}\vec b_{k+d-1},
\\
      0<\nu_i-3\dfrac{|\underline{\vec z_{i+2}}|}{|\vec b_i|_{\vphantom{|}}}\leq1\hskip18.5mm\text{for }i\leq k+d-3, \\
      0<\nu_i-3\dfrac{|\underline{\vec z_{k+d-1}}|^{1+\beta}}{|\vec b_i|}\leq1\qquad\text{for }i\geq k+d-2.
    \end{array}
  \end{equation}

  \paragraph{Providing the hypothesis of Lemma \ref{l:basic_lemma}.}
  Let us prove the following statements:
  \\
  \textup{(a$'$)} $\vec z_{k+1},\ldots,\vec z_{k+d}$ form a basis of $\Z^d$; \\
  \textup{(b$'$)} $\langle\underline{\vec z_{k+d-1}},\underline{\vec z_{k+d}}\rangle<0$; \\
  \textup{(c$'$)} $B_{k+d-1},D_{k+d-1,2}>3$, \\
  \phantom{\textup{(c$'$)}} $B_{k+d-1}\asymp D_{k+d-1,2}\asymp|\underline{\vec z_{k+d-1}}|^\beta$; \\
  \textup{(d$'$)}           $\det\underline\La_{k+d+1-n,n}>3|\underline{\vec z_{k+d+2-n}}|\det\underline\La_{k+d+1-n,n-1}$, \\
  \phantom{\textup{(d$'$)}} $\det\underline\La_{k+d+1-n,n}\asymp\ \,|\underline{\vec z_{k+d+2-n}}|\det\underline\La_{k+d+1-n,n-1},\ \ \ n=3,\ldots,d-1$ (if $d\geq4$); \\
  \textup{(e$'$)} $\langle\pmb\alpha_{k+2},\vec z_{k+1}\rangle\cdot\langle\pmb\alpha_{k+2},\vec z_{k}\rangle>0$.

  Statement \textup{(a$'$)} follows immediately from $\vec z_{k+d}\in-\vec z_{k}+\La_{k+1,d-1}$.

  By \eqref{eq:the_nu's} we have
  \begin{multline*}
    \langle\underline{\vec z_{k+d-1}},\underline{\vec z_{k+d}}\rangle=
    \sum_{i=k+1}^{k+d-2}\nu_i\langle\vec b_i,\vec b_{k+d-1}\rangle-\nu_{k+d-1}\langle\vec b_{k+d-1},\vec b_{k+d-1}\rangle= \\ =
    -\nu_{k+d-1}|\vec b_{k+d-1}|^2<0,
  \end{multline*}
  which provides \textup{(b$'$)}.

  By \eqref{eq:z_i_leq_z_i+1} and \eqref{eq:the_nu's} there exist $\lambda_k',\lambda_k'',\lambda_k'''\in(1,2\sqrt d)$ such that
  \begin{equation} \label{eq:lambda'}
    |\underline{\vec z_{k+d}}|=
    \bigg(\sum_{i=k+1}^{k+d-1}\nu_i^2|\vec b_i|^2\bigg)^{1/2}=
    \lambda_k'\cdot3|\underline{\vec z_{k+d-1}}|^{1+\beta};
  \end{equation}
  \begin{equation} \label{eq:lambda''}
  \begin{aligned}
    |\underline{\vec z_{k+d-1}}\wedge\underline{\vec z_{k+d}}| & =
    \bigg|\vec b_{k+d-1}\wedge\sum_{i=k+1}^{k+d-2}\nu_i\vec b_i\bigg|= \\ & =
    |\vec b_{k+d-1}|\bigg(\sum_{i=k+1}^{k+d-2}\nu_i^2|\vec b_i|^2\bigg)^{1/2}=
    \lambda_k''\cdot3|\underline{\vec z_{k+d-1}}|^{2+\beta};
    \hskip2.1mm
  \end{aligned}
  \end{equation}
  \begin{equation*}
  \begin{aligned}
    |\underline{\vec z_{k+d+1-n}}\wedge\ldots\wedge\underline{\vec z_{k+d}}| & =
    \bigg|\vec b_{k+d+1-n}\wedge\ldots\wedge\vec b_{k+d-1}\wedge\sum_{i=k+1}^{k+d-n}\nu_i\vec b_i\bigg|= \\ & =
    |\vec b_{k+d+1-n}\wedge\ldots\wedge\vec b_{k+d-1}|\bigg(\sum_{i=k+1}^{k+d-n}\nu_i^2|\vec b_i|^2\bigg)^{1/2}=
    \hskip15.8mm
    \\ & =
    \vphantom{\bigg|}
    \lambda_k'''\cdot3|\vec z_{k+d+2-n}|\cdot|\underline{\vec z_{k+d+1-n}}\wedge\ldots\wedge\underline{\vec z_{k+d-1}}| \\ &
    \hskip-39.3mm
    \vphantom{\bigg|}
    \text{for $n=3,\ldots,d-1$ (if $d\geq4$)}.
  \end{aligned}
  \end{equation*}
  This provides \textup{(c$'$)} and \textup{(d$'$)}.

  To prove \textup{(e$'$)} let us define $\mu_{k+1},\ldots,\mu_{k+d-1}$ by
  \begin{equation} \label{eq:the_mu's}
    \vec z_{k+d}=-\vec z_{k}+\sum_{i=k+1}^{k+d-1}\mu_i\vec z_i.
  \end{equation}
  We have $\langle\pmb\alpha_{k+2},\vec z_i\rangle=0$ for each $i=k+2,\ldots,k+d$, so,
  \[0=\langle\pmb\alpha_{k+2},\vec z_{k+d}\rangle=-\langle\pmb\alpha_{k+2},\vec z_{k}\rangle+\mu_{k+1}\langle\pmb\alpha_{k+2},\vec z_{k+1}\rangle,\]
  whence
  \[\langle\pmb\alpha_{k+2},\vec z_{k+1}\rangle\cdot\langle\pmb\alpha_{k+2},\vec z_{k}\rangle=\mu_{k+1}\langle\pmb\alpha_{k+2},\vec z_{k+1}\rangle^2.\]
  Since $\langle\pmb\alpha_{k+2},\vec z_{k+1}\rangle\neq0$, it suffices to show that $\mu_{k+1}>0$. By \eqref{eq:the_mu's} we have
  \[\langle\vec b_{k+1},\underline{\vec z_{k+d}}\rangle=
    \langle\vec b_{k+1},-\underline{\vec z_{k}}+\sum_{i=k+1}^{k+d-1}\mu_i\underline{\vec z_i}\rangle=
    -\langle\vec b_{k+1},\underline{\vec z_{k}}\rangle+\mu_{k+1}|\vec b_{k+1}|^2,\]
  while by \eqref{eq:the_nu's} we have
  \[\langle\vec b_{k+1},\underline{\vec z_{k+d}}\rangle=
    \langle\vec b_{k+1},\sum_{i=k+1}^{k+d-2}\nu_i\vec b_i-\nu_{k+d-1}\vec b_{k+d-1}\rangle=
    \nu_{k+1}|\vec b_{k+1}|^2.\]
  Hence, taking into account that $\nu_{k+1}>3|\underline{\vec z_{k+2}}|/|\vec b_{k+1}|$
  regardless of whether $d=3$ or $d\geq4$, we get
  \[\mu_{k+1}=\nu_{k+1}+
    \frac{\langle\vec b_{k+1},\underline{\vec z_{k}}\rangle}{|\vec b_{k+1}|^2}>
    3\frac{|\underline{\vec z_{k+2}}|}{|\vec b_{k+1}|}+
    \frac{\langle\vec b_{k+1},\underline{\vec z_{k}}\rangle}{|\vec b_{k+1}|^2}>
    \frac{|\vec b_{k+1}|\cdot|\underline{\vec z_{k}}|+\langle\vec b_{k+1},\underline{\vec z_{k}}\rangle}{|\vec b_{k+1}|^2}\geq0.\]

  Thus, \textup{(a$'$)}--\textup{(e$'$)} are valid, which implies that the sequence $(\vec z_k)_{k\in\N}$ satisfies the hypothesis of Lemma \ref{l:basic_lemma}, and we can set
  $\pmb\alpha=\lim_{k\to\infty}\pmb\alpha_k$.
%
%

  \paragraph{Asymptotic directions.}
  Let us prove that both sequences
  \begin{equation} \label{eq:even_and_odd}
    \bigg(\frac{\underline{\vec z_{k}}}{|\underline{\vec z_{k}}|}\bigg)_{k\equiv0\hskip-2mm\pmod2} 
    \quad\text{ and }\quad
    \bigg(\frac{\underline{\vec z_{k}}}{|\underline{\vec z_{k}}|}\bigg)_{k\equiv1\hskip-2mm\pmod2} 
  \end{equation}
  converge. Since
  \begin{equation} \label{eq:distance_via_inner_product}
    \bigg|\frac{\underline{\vec z_{k-2}}}{|\underline{\vec z_{k-2}}|}-\frac{\underline{\vec z_{k}}}{|\underline{\vec z_{k}}|}\bigg|^2=
    2\bigg(1-\frac{\langle\underline{\vec z_{k-2}},\underline{\vec z_{k}}\rangle}{|\underline{\vec z_{k-2}}|\cdot|\underline{\vec z_{k}}|}\bigg),
  \end{equation}
  we are to estimate $\langle\underline{\vec z_{k-2}},\underline{\vec z_{k}}\rangle$.

  Using the inductive definition of $\vec z_{k+d}$ along with relations \eqref{eq:z_i_leq_z_i+1}, \eqref{eq:the_nu's} one can easily prove the following asymptotic analogues of \eqref{eq:lambda'} and \eqref{eq:lambda''}:
  \begin{equation} \label{eq:asymp_instead_of_lambdas}
  \begin{aligned}
    |\underline{\vec z_{k+d}}| & =
    3\sqrt2|\underline{\vec z_{k+d-1}}|^{1+\beta}+O\big(|\underline{\vec z_{k+d-1}}|\big), \\
    |\underline{\vec z_{k+d-1}}\wedge\underline{\vec z_{k+d}}| & =
    3|\underline{\vec z_{k+d-1}}|^{2+\beta}+O\big(|\underline{\vec z_{k+d-1}}|^2\big).
  \end{aligned}
  \end{equation}
  Similarly,
  \begin{equation*}
  \begin{aligned}
    \hskip1.8mm
    \langle\underline{\vec z_{k+d-1}},\underline{\vec z_{k+d}}\rangle & =
    -\nu_{k+d-1}|\vec b_{k+d-1}|^2= \\ & =
    -3|\underline{\vec z_{k+d-1}}|^{2+\beta}+O\big(|\underline{\vec z_{k+d-1}}|^2\big).
  \end{aligned}
  \end{equation*}
  Hence
  \begin{multline*}
    \hskip-3.5mm
    \langle\underline{\vec z_{k+d-2}},\underline{\vec z_{k+d}}\rangle=
    \langle\underline{\vec z_{k+d-2}},\nu_{k+d-2}\vec b_{k+d-2}-\nu_{k+d-1}\vec b_{k+d-1}\rangle= \\ \hskip1mm =
    3|\underline{\vec z_{k+d-1}}|^{1+\beta}
    \bigg(\frac{\langle\underline{\vec z_{k+d-2}},\vec b_{k+d-2}\rangle}{|\vec b_{k+d-2}|}-
          \frac{\langle\underline{\vec z_{k+d-2}},\vec b_{k+d-1}\rangle}{|\vec b_{k+d-1}|}\bigg)+
    O\big(|\underline{\vec z_{k+d-2}}|\cdot|\underline{\vec z_{k+d-1}}|\big)= \\ \hskip7mm =
    3|\underline{\vec z_{k+d-1}}|^{\beta}
    \Big(|\underline{\vec z_{k+d-2}}\wedge\underline{\vec z_{k+d-1}}|-
         \langle\underline{\vec z_{k+d-2}},\underline{\vec z_{k+d-1}}\rangle\Big)+
    O\big(|\underline{\vec z_{k+d-2}}|\cdot|\underline{\vec z_{k+d-1}}|\big)= \\ \hskip33mm =
    3|\underline{\vec z_{k+d-1}}|^{\beta}
    \Big(6|\underline{\vec z_{k+d-2}}|^{2+\beta}
    +O\big(|\underline{\vec z_{k+d-2}}|^2\big)\Big)+
    O\big(|\underline{\vec z_{k+d-2}}|^{2+\beta}\big)= \vphantom{\bigg|} \\ =
    18|\underline{\vec z_{k+d-1}}|^{\beta}
    \Big(|\underline{\vec z_{k+d-2}}|^{2+\beta}
    +O\big(|\underline{\vec z_{k+d-2}}|^2\big)\Big),
  \end{multline*}
  i.e.
  \begin{multline*}
    \frac{\langle\underline{\vec z_{k+d-2}},\underline{\vec z_{k+d}}\rangle}{|\underline{\vec z_{k+d-2}}|\cdot|\underline{\vec z_{k+d}}|}=
    18\frac{|\underline{\vec z_{k+d-1}}|^{1+\beta}}{|\underline{\vec z_{k+d}}|}\cdot
    \frac{|\underline{\vec z_{k+d-2}}|^{1+\beta}}{|\underline{\vec z_{k+d-1}}|}\cdot
    \big(1+O\big(|\underline{\vec z_{k+d-2}}|^{-\beta}\big)\big)= \\ =
    \big(1+O\big(|\underline{\vec z_{k+d-1}}|^{-\beta}\big)\big)\big(1+O\big(|\underline{\vec z_{k+d-2}}|^{-\beta}\big)\big)^2=
    1+O\big(|\underline{\vec z_{k+d-2}}|^{-\beta}\big).
  \end{multline*}
  Returning to \eqref{eq:distance_via_inner_product} we get
  \[\bigg|\frac{\underline{\vec z_{k}}}{|\underline{\vec z_{k}}|}-\frac{\underline{\vec z_{k-2}}}{|\underline{\vec z_{k-2}}|}\bigg|^2=
    O\big(|\underline{\vec z_{k-2}}|^{-\beta}\big)=O(k^{-\beta}).\]
  Thus, indeed, both sequences \eqref{eq:even_and_odd} do converge. Moreover, the corresponding limit points are distinct, as by \eqref{eq:asymp_instead_of_lambdas} we have
  \[\lim_{k\to\infty}\frac{\langle\underline{\vec z_{k-1}},\underline{\vec z_{k}}\rangle}{|\underline{\vec z_{k-1}}|\cdot|\underline{\vec z_{k}}|}=\frac1{\sqrt2}\,.\]
  It remains to notice that since $L_{\pmb\alpha}(\vec z_k)\to0$ as $k\to\infty$, the convergence of the sequences \eqref{eq:even_and_odd} implies the convergence of
  \begin{equation*}
    \bigg(\frac{\vec z_{k}}{|\vec z_{k}|}\bigg)_{k\equiv0\hskip-2mm\pmod2} 
    \quad\text{ and }\quad
    \bigg(\frac{\vec z_{k}}{|\vec z_{k}|}\bigg)_{k\equiv1\hskip-2mm\pmod2}. 
  \end{equation*}
\end{proof}

\paragraph{Acknowledgements.}

The author is grateful to N.\,G.\,Moshchevitin, N.\,Technau and M.\,Widmer for fruitful discussions.

%
%
%
%
%


\begin{thebibliography}{99}

\bibitem
    {cassels_GN}
    \textsc{J.\,W.\,S.\,Cassels}
    \textit{An introduction to the geometry of numbers.}
    Springer (1997).
\bibitem
    {borevich_shafarevich}
    \textsc{Z.\,I.\,Borevich, I.\,R.\,Shafarevich}
    \textit{Number theory}.
    NY Academic Press (1966).
\bibitem
    {schmidt_subspace}
    \textsc{W.\,M.\,Schmidt}
    \textit{Norm form equations.}
    Ann. Math., \textbf{96} (1972), 526--551.
\bibitem
    {skriganov_1998}
    \textsc{M.\,M.\,Skriganov}
    \textit{Ergodic theory on $\SL(n)$, Diophantine approximations and anomalies in the lattice point problem.}
    Invent. math. \textbf{132} (1998), 1--72.
\bibitem
    {german_2017}
    \textsc{O.\,N.\,German}
    \textit{Diophantine exponents of lattices.}
    Proc. Steklov Inst. Math. \textbf{296}, suppl. 2 (2017), 29--35.
\bibitem
    {kleinbock_margulis_1999}
    \textsc{D.\,Y.\,Kleinbock, G.\,A.\,Margulis}
    \textit{Logarithm laws for flows on homogeneous spaces.}
    Invent. math. \textbf{138} (1999), 451--494.
\bibitem
    {technau_widmer}
    \textsc{N.\,Technau, M.\,Widmer}
    \textit{On a counting theorem of Skriganov.}
     Preprint, arXiv:1611.02649.
\bibitem
    {german_moshchevitin_bordeaux}
    \textsc{O.\,N.\,German, N.\,G.\,Moshchevitin}
    \textit{Linear forms of a given Diophantine type.}
    JTNB {\bf 22}:2 (2010), 383--396.
\bibitem
    {german_2004}
    \textsc{O.\,N.\,German}
    \textit{Asymptotic directions for best approximations of n-dimensional linear forms.}
    Math. Notes \textbf{75}:1-2 (2004), 51--65.
\bibitem
    {schmidt_summerer_2009}
    \textsc{W.\,M.\,Schmidt, L.\,Summerer}
    \textit{Parametric geometry of numbers and applications.}
    Acta Arithmetica \textbf{140} (2009), 67--91.
\bibitem
    {schmidt_summerer_2013}
    \textsc{W.\,M.\,Schmidt, L.\,Summerer}
    \textit{Diophantine approximation and parametric geometry of numbers.}
    Monatsh. Math. \textbf{169} (2013), 51--104.
\bibitem
    {roy_annals_2015}
    \textsc{D.\,Roy}
    \textit{On Schmidt and Summerer parametric geometry of numbers.}
    Ann. Math. \textbf{182}:2 (2015), 739--786.


%
%
\end{thebibliography}
\end{document}